\documentclass[11pt]{article}

\usepackage[T1]{fontenc}
\usepackage[utf8]{inputenc}
\usepackage{lmodern}
\usepackage{microtype}
\usepackage{amsmath,amssymb,amsthm,mathtools}
\usepackage{aliascnt}
\usepackage{enumitem}
\usepackage{geometry}
\usepackage{booktabs}
\usepackage{array}
\usepackage{hyperref}
\usepackage[nameinlink,noabbrev]{cleveref}
\usepackage{xcolor}
\usepackage{tikz}
\usepackage{float}

\allowdisplaybreaks
\hypersetup{
  hypertexnames=false,
  colorlinks=true,
  linkcolor=blue!45!black,
  citecolor=blue!45!black,
  urlcolor=blue!45!black,
  pdftitle={Group-product rigidity and Higman--Thompson reassociation groups},
  pdfauthor={Arthur Queiroz Moura}
}

\newtheorem{theorem}{Theorem}[section]
\newaliascnt{proposition}{theorem}
\newtheorem{proposition}[proposition]{Proposition}
\aliascntresetthe{proposition}
\newaliascnt{lemma}{theorem}
\newtheorem{lemma}[lemma]{Lemma}
\aliascntresetthe{lemma}
\newaliascnt{corollary}{theorem}
\newtheorem{corollary}[corollary]{Corollary}
\aliascntresetthe{corollary}
\theoremstyle{definition}
\newaliascnt{definition}{theorem}

\aliascntresetthe{definition}
\newaliascnt{example}{theorem}
\newtheorem{example}[example]{Example}
\aliascntresetthe{example}
\theoremstyle{remark}
\newaliascnt{remark}{theorem}
\newtheorem{remark}[remark]{Remark}
\aliascntresetthe{remark}

\newcommand{\Aut}{\operatorname{Aut}}
\newcommand{\Inn}{\operatorname{Inn}}
\newcommand{\Out}{\operatorname{Out}}
\newcommand{\Hol}{\operatorname{Hol}}
\newcommand{\Mlt}{\operatorname{Mlt}}
\newcommand{\End}{\operatorname{End}}
\newcommand{\Ad}{\operatorname{Ad}}
\newcommand{\Assoc}{\operatorname{Assoc}}
\newcommand{\ord}{\operatorname{ord}}
\newcommand{\id}{\operatorname{id}}
\newcommand{\Int}{\operatorname{Int}}
\newcommand{\Z}{\mathbb Z}

\newcommand{\cD}{\mathcal D}

\title{\bfseries Group-product rigidity and\\Higman--Thompson reassociation groups}
\author{Arthur Queiroz Moura\\
\small Independent researcher, Brazil\\
\small \texttt{edu@arthurqm.com}}
\date{}

\begin{document}
\maketitle

\begin{abstract}
Fix an arity $r\ge 2$, a group $G$, and a full ordered $r$-ary tree $T$ with at least two internal vertices. For an arbitrary operation $\omega:G^r\to G$, let $\omega_T$ denote the operation obtained by iterating $\omega$ according to $T$. We classify all $\omega$ for which there exists a bijection $F_T:G\to G$ such that
\[
\omega_T(x_1,\ldots,x_n)=F_T(x_1\cdots x_n).
\]
We prove that $\omega$ must have one of the forms
\[
ax_1\cdots x_r,\qquad
x_1\cdots x_rb,\qquad
d\,\psi(x_1\cdots x_r),
\]
where $a,b\in G$, $d\in Z(G)$, and $\psi\in\operatorname{Aut}(G)$, with explicit conditions on the parameters determined by $T$.

We next reverse the problem. Fix an operation $\omega$ of one of these three forms and determine every full ordered $r$-ary tree $T$ for which there exists a bijection $F_T:G\to G$ satisfying $\omega_T(x_1,\ldots,x_n)=F_T(x_1\cdots x_n)$. The answer is governed by the order of $aZ(G)$ or $bZ(G)$ in $G/Z(G)$, or by the order of $\psi$ in $\operatorname{Aut}(G)$.

We also ask, for each of the three solution forms above, how much associativity remains. More precisely, for two full ordered $r$-ary trees $S$ and $T$ with the same number of leaves, we determine exactly when $\omega_S=\omega_T$. Pairs of $r$-ary trees encode changes of parenthesization, and modulo simultaneous expansion they represent elements of the Higman--Thompson group $F_r$. The changes of parenthesization that preserve the iterated operation form a subgroup of $F_r$, which we determine explicitly.
\end{abstract}



\section{Introduction}\label{sec:introduction}

Fix an integer $r\ge2$ and a group $G$.  A full ordered $r$-ary tree is a rooted ordered tree in which every internal vertex has exactly $r$ children.  Such a tree records a parenthesization of an iterated $r$-ary operation $\omega\colon G^r\to G$: put the inputs at the leaves and apply $\omega$ at each internal vertex.  If $T$ has $n$ leaves, we write $\omega_T$ for the resulting $n$-variable operation.  For example, for a binary tree representing $x*(y*z)$, one has $\omega_T(x,y,z)=\omega(x,\omega(y,z))$.

We first fix the tree and solve for the operation.  Let $T$ be a full ordered $r$-ary tree with at least two internal vertices.  We ask for all $\omega$ for which there is a bijection $F_T\colon G\to G$ satisfying
\begin{equation}\label{eq:intro-factorization}
  \omega_T(x_1,\ldots,x_n)=F_T(x_1\cdots x_n).
\end{equation}
Thus the chosen iterate is allowed to depend on the inputs only through their ordered product in $G$.  The main rigidity theorem shows that this condition forces $\omega$ into one of three forms:
\begin{equation}\label{eq:intro-three-families}
  a x_1\cdots x_r,\qquad
  x_1\cdots x_r b,\qquad
  d\,\psi(x_1\cdots x_r),
\end{equation}
where $a,b\in G$, $d\in Z(G)$, and $\psi\in\Aut(G)$.  The fixed tree determines which parameters are allowed.  Conversely, every operation belonging to one of these three families satisfies \eqref{eq:intro-factorization} for some tree with at least two internal vertices.  Thus these are exactly the operations that can arise from the fixed-tree problem as the tree varies.

For $r=2$, the three forms are already visible in the three height-two trees in \cref{fig:binary-height-two}.  A \emph{right $r$-ary vine of length $k$} has $k$ internal vertices arranged in a chain.  The root is the first internal vertex; at each internal vertex except the last, only the rightmost child is internal, while every other child is a leaf.  The last internal vertex has only leaf children.  A left vine is the mirror image.

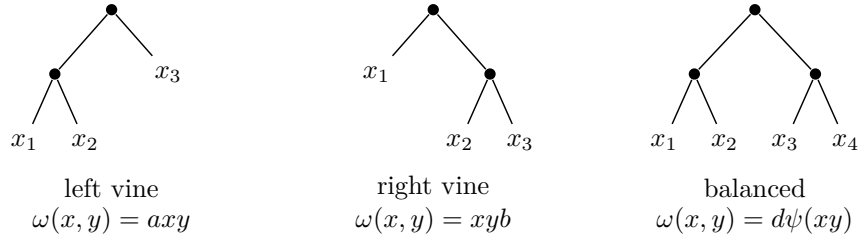
\begin{figure}[H]
\centering
\begin{tikzpicture}[line width=.55pt, font=\small]
  \begin{scope}[xshift=0cm]
    \node[circle,fill=black,inner sep=1.5pt] (L0) at (0,2.0) {};
    \node[circle,fill=black,inner sep=1.5pt] (L1) at (-.75,1.15) {};
    \node (L3) at (.75,1.15) {$x_3$};
    \node (Lx1) at (-1.15,.25) {$x_1$};
    \node (Lx2) at (-.35,.25) {$x_2$};
    \draw (L0)--(L1) (L0)--(L3) (L1)--(Lx1) (L1)--(Lx2);
    \node at (0,-.35) {left vine};
    \node at (0,-.78) {$\omega(x,y)=a xy$};
  \end{scope}
  \begin{scope}[xshift=4.25cm]
    \node[circle,fill=black,inner sep=1.5pt] (R0) at (0,2.0) {};
    \node (Rx1) at (-.75,1.15) {$x_1$};
    \node[circle,fill=black,inner sep=1.5pt] (R1) at (.75,1.15) {};
    \node (Rx2) at (.35,.25) {$x_2$};
    \node (Rx3) at (1.15,.25) {$x_3$};
    \draw (R0)--(Rx1) (R0)--(R1) (R1)--(Rx2) (R1)--(Rx3);
    \node at (0,-.35) {right vine};
    \node at (0,-.78) {$\omega(x,y)=xyb$};
  \end{scope}
  \begin{scope}[xshift=8.5cm]
    \node[circle,fill=black,inner sep=1.5pt] (B0) at (0,2.0) {};
    \node[circle,fill=black,inner sep=1.5pt] (B1) at (-.8,1.15) {};
    \node[circle,fill=black,inner sep=1.5pt] (B2) at (.8,1.15) {};
    \node (Bx1) at (-1.2,.25) {$x_1$};
    \node (Bx2) at (-.4,.25) {$x_2$};
    \node (Bx3) at (.4,.25) {$x_3$};
    \node (Bx4) at (1.2,.25) {$x_4$};
    \draw (B0)--(B1) (B0)--(B2) (B1)--(Bx1) (B1)--(Bx2)
          (B2)--(Bx3) (B2)--(Bx4);
    \node at (0,-.35) {balanced};
    \node at (0,-.78) {$\omega(x,y)=d\psi(xy)$};
  \end{scope}
\end{tikzpicture}
\caption{For $r=2$, these are the three height-two shapes with at least one
internal child.  If \eqref{eq:intro-factorization} holds for the displayed
tree, the three shapes force, respectively, a left translation, a right
translation, and the form $d\psi(xy)$ with $d\in Z(G)$ and
$\psi\in\Aut(G)$.}
\label{fig:binary-height-two}
\end{figure}

We next reverse the problem.  Fix an operation $\omega$ of one of the three forms in \eqref{eq:intro-three-families}, and now let the tree vary.  We determine every full ordered $r$-ary tree $T$ for which there is a bijection $F_T\colon G\to G$ satisfying $\omega_T(x_1,\ldots,x_n)=F_T(x_1\cdots x_n)$.  For a right translation $\omega(x_1,\ldots,x_r)=x_1\cdots x_r b$, the answer is controlled by the order of $bZ(G)$ in $G/Z(G)$.  If this order is a finite integer $k$, the allowed trees are obtained by independently inserting right vines of length $k$ at chosen leaves of a shorter right vine.  If the order is infinite, this periodic freedom disappears and the whole tree must itself be a right vine.  The left-translation case is the mirror image.  For $\omega(x_1,\ldots,x_r)=d\psi(x_1\cdots x_r)$, finite order $h=\ord(\psi)$ permits independent insertion of perfect $r$-ary trees of height $h$, where a perfect tree is one whose leaves all have the same depth.  If $\psi$ has infinite order, all leaf depths must be equal and the whole tree is perfect.  Central translations work on every tree.

We also ask, for each of the three solution forms in \eqref{eq:intro-three-families}, how much associativity remains.  These operations are built from an associative group product but are generally not associative themselves.  For example, if $r=2$ and $\omega(x,y)=xyb$, then $\omega(\omega(x,y),z)=xybzb$, whereas $\omega(x,\omega(y,z))=xyzb^2$; these need not agree.  We determine which changes of parenthesization preserve the iterated operation: for two full ordered $r$-ary trees $S$ and $T$ with the same number of leaves, when does $\omega_S=\omega_T$?  This is a separate, related question.  Neither tree is required to have an iterate that depends only on the ordered group product; we are now asking for all identities between different parenthesizations of the fixed operation.

The Higman--Thompson group $F_r$ organizes these changes of parenthesization.  An element of $F_r$ is represented by a pair $(S,T)$ of full ordered $r$-ary trees with the same number of leaves.  Simultaneously expanding corresponding leaves of both trees does not change the represented element, and composition corresponds to composing the changes of parenthesization.  For a fixed $\omega$, the pairs satisfying $\omega_S=\omega_T$ form a subgroup of $F_r$.

We compute this subgroup for each of the three forms.  If $\omega(x_1,\ldots,x_r)=x_1\cdots x_r b$ and $bZ(G)$ has finite order $k$, a right $r$-ary vine of length $k$ has $m=1+k(r-1)$ leaves.  Replacing every internal vertex of an $m$-ary tree by this vine gives an embedding $F_m\hookrightarrow F_r$, and the subgroup of surviving reassociations is exactly its image.  The left-translation case uses the left vine.  If $\omega=d\psi(x_1\cdots x_r)$ and $\psi$ has finite order $h$, replacing each $r^h$-ary internal vertex by a perfect $r$-ary tree of height $h$ gives a copy of $F_{r^h}$.  Infinite order leaves no nontrivial reassociation.  Thus the same finite orders that control the individual-tree problem also determine the amount of associativity that survives.

As $G$ and $\omega$ vary, the nontrivial abstract reassociation groups that occur are exactly $F_m$ with $m\ge r$ and $m\equiv1\pmod{r-1}$.  We also distinguish the particular copies of these groups that arise inside $F_r$.

\Cref{sec:crossed-module} gives an optional explanation of the three formulas in \eqref{eq:intro-three-families} using the crossed module $G\xrightarrow{\Ad}\Aut(G)$.  It is not used in the classification proofs.

\subsection*{Relation with earlier work}
Generalized associativity and quasigroup functional equations go back to
Acz\'el--Belousov--Hossz\'u \cite{AczelBelousovHosszu1960}, Belousov
\cite{Belousov1966}, and later work such as Krape\v z--Taylor
\cite{KrapezTaylor1991,KrapezTaylor1995}.  Multary linear quasigroups and
autotopies are treated by Marini--Shcherbacov \cite{MariniShcherbacov2004}.
Dr\'apal studies binary group isotopes using a holomorphic action
\cite{Drapal2009}.  Zaslavsky studies factorizations of multary quasigroups
through lower-arity operations using biased expansions \cite{Zaslavsky2012},
while the Hossz\'u--Gluskin theory describes totally associative $n$-ary
groups in terms of ordinary groups and automorphisms; see Dudek--Gl\l azek
\cite{DudekGlazek2008}.  Our first problem is different: the comparison
operation is the fixed group product, the tree $T$ is fixed, and $\omega$ is
arbitrary before \eqref{eq:intro-factorization} is imposed.

Associative spectra were introduced by Cs\'ak\'any--Waldhauser
\cite{CsakanyWaldhauser2000}.  Lehtonen--Waldhauser classify bracketing
identities of binary linear quasigroups using left/right depths and lattice
data \cite{LehtonenWaldhauser2023}.  Geoghegan--Guzm\'an use Thompson's group
$F$ to encode stable associativity \cite{GeogheganGuzman2006}.

The subgroup constructions used below also have precedents.  Golan--Sapir
identify the binary digit-sum subgroups and tree-substitution maps
\cite{GolanSapir2017}.  Aiello--Nagnibeda define $H_r\le F_r$ by replacing
one $(2r-1)$-ary internal vertex with a two-vertex right $r$-ary vine
\cite[Sec.~5]{AielloNagnibeda2022}; this is the case in our translation
theorem where the element of $G/Z(G)$ has order $2$.  Recent work on $F_n$
uses cores, digit-sum residues modulo $n-1$, and transducers
\cite{GolanSapir2026,Golan2026}.  We do not claim these constructions or
residue invariants as new.  Here the additional modulus $k$ is the order of
$aZ(G)$ or $bZ(G)$ in $G/Z(G)$, and we prove that this group-theoretic period
controls all surviving reassociations of the corresponding operation.

\Cref{sec:rigidity} proves the algebraic rigidity behind the fixed-tree problem.  \Cref{sec:transport} derives exact formulas for the three resulting families; with $T$ fixed these formulas finish the classification of $\omega$, and with $\omega$ fixed they classify all trees satisfying the product equation.  \Cref{sec:crossed-module} gives the optional structural explanation.  \Cref{sec:reassociation} then turns to the separate associativity question and determines all identities $\omega_S=\omega_T$ and the corresponding subgroups of $F_r$.  \Cref{sec:consequences} records two further consequences.
\section{From one fixed tree to the algebraic form}\label{sec:rigidity}

We first prove a reduction that does not use the group structure.  A
\emph{full ordered $r$-ary tree} is either a leaf or a root with an ordered
$r$-tuple of full ordered $r$-ary subtrees.  Let $N(T)$ and $L(T)$ be the
numbers of internal vertices and leaves.  Then
\begin{equation}\label{eq:leaf-count}
  L(T)=1+(r-1)N(T).
\end{equation}
We call a tree \emph{nontrivial} if it has an internal vertex.  The
$r$-corolla is the tree with one internal vertex and $r$ leaves.  A
\emph{fringe subtree} is the subtree consisting of a vertex and all of its
descendants.

For an $r$-ary operation $\eta\colon Q^r\to Q$, the notation $\eta_T$ means
the operation obtained by placing variables at the leaves of $T$ and applying
$\eta$ at every internal vertex.  We number the leaves from left to right.

Recall that $\nu\colon Q^r\to Q$ is an $r$-ary quasigroup if fixing any
$r-1$ coordinates leaves a bijection in the remaining coordinate.  The same
is then true of every tree term $\nu_T$ in each leaf variable.

\begin{theorem}[Relative one-tree rigidity]\label{thm:relative-one-tree}
Let $\nu\colon Q^r\to Q$ be an $r$-ary quasigroup and let
$\omega\colon Q^r\to Q$ be arbitrary.  Suppose that for one nontrivial full
ordered $r$-ary tree $T$ there is a bijection $F_T\colon Q\to Q$ such that
\[
  \omega_T=F_T\circ\nu_T.
\]
Then there is a unique bijection $f\colon Q\to Q$ such that
\begin{equation}\label{eq:relative-output-isotope}
  \omega=f\circ\nu.
\end{equation}
Moreover, for every nontrivial fringe subtree $S$ of $T$ there is a bijection
$F_S\colon Q\to Q$ such that $\omega_S=F_S\circ\nu_S$.
\end{theorem}

\begin{proof}
If $T$ is the $r$-corolla, the first assertion is immediate.  Suppose that
$T$ has at least two internal vertices and choose an internal vertex $v$ whose children are all leaves.  We first show that $\omega$ is injective in every coordinate.  Fix all
leaves except one leaf below $v$.  If two values of that leaf give the same value at $v$, then the final
$\omega_T$-values agree.  Since
$\omega_T=F_T\nu_T$ and $F_T$ is injective, the corresponding $\nu_T$-values
agree.  The term $\nu_T$ is bijective in the chosen leaf coordinate, so the
two values coincide.  Hence, whenever all but one input of $\omega$ are fixed, varying the remaining
input gives an injective map.  The same is true for any unary map obtained by
fixing every leaf outside one chosen subtree.

Fix the leaves outside $v$.  Inserting the value produced at $v$ into the
rest of the fixed tree defines a unary map $C\colon Q\to Q$.  Inside $v$, fix all local coordinates except one and vary
the remaining coordinate.  The value $F_T\nu_T$ ranges through all of $Q$,
because $\nu_T$ is bijective in that leaf and $F_T$ is bijective.  Hence $C$
is surjective, and therefore bijective.

Now vary all $r$ inputs at $v$.  After the leaves outside $v$ are fixed, the rest of the $\nu$-tree likewise
defines a bijection $D\colon Q\to Q$.  Thus
\[
  C\bigl(\omega(x_1,\ldots,x_r)\bigr)
  =F_T\bigl(D(\nu(x_1,\ldots,x_r))\bigr).
\]
It follows that $\omega=f\circ\nu$ with $f=C^{-1}F_TD$.  Uniqueness follows
from the surjectivity of $\nu$.

Equation \eqref{eq:relative-output-isotope} makes $\omega$ an $r$-ary
quasigroup.  Let $S$ be a nontrivial fringe subtree.  Fix the leaves outside
$S$.  The two unary maps defined by the fixed part of the $\omega$- and $\nu$-trees
are bijections, so applying their inverses to $\omega_T=F_T\nu_T$ gives
$\omega_S=F_S\nu_S$ for a bijection $F_S$.
\end{proof}

We now take $\nu$ to be the $r$-fold group product
$m_r(x_1,\ldots,x_r)=x_1\cdots x_r$.  Associativity gives
$(m_r)_T(x_1,\ldots,x_n)=x_1\cdots x_n$ for every $T$.  Thus the condition of interest is that the map $\omega_T$ factor through the
ordered product map, meaning that
$\omega_T(x_1,\ldots,x_n)=F_T(x_1\cdots x_n)$ for some bijection
$F_T\colon G\to G$.

\begin{corollary}[One-tree reduction]\label{cor:one-tree-reduction}
Let $T$ be a nontrivial full ordered $r$-ary tree, and suppose there is a
bijection $F_T\colon G\to G$ such that
\[
  \omega_T(x_1,\ldots,x_n)=F_T(x_1\cdots x_n).
\]
Then there is a unique bijection $f\colon G\to G$ such that
\begin{equation}\label{eq:output-isotope}
  \omega(x_1,\ldots,x_r)=f(x_1\cdots x_r).
\end{equation}
Moreover, for every nontrivial fringe subtree $S$ of $T$, there is a
bijection $F_S\colon G\to G$ with
\[
  \omega_S(x_1,\ldots,x_m)=F_S(x_1\cdots x_m),
\]
where $m$ is the number of leaves of $S$.
\end{corollary}

\begin{proof}
Apply \cref{thm:relative-one-tree} to $\nu=m_r$.
\end{proof}

The next lemma records the three two-variable equations that arise from a
height-two subtree.

\begin{lemma}[Three two-variable equations]\label{lem:pexider}
Let $f\colon G\to G$ be a bijection.
\begin{enumerate}[label=\textup{(\roman*)}]
\item If $f(x)y=K(xy)$ for some map $K\colon G\to G$ and all $x,y\in G$,
then $f(x)=ax$ for a unique $a\in G$.
\item If $xf(y)=K(xy)$ for some map $K$ and all $x,y\in G$, then
$f(y)=yb$ for a unique $b\in G$.
\item If $f(x)f(y)=K(xy)$ for some map $K$ and all $x,y\in G$, then there
are unique $d\in Z(G)$ and $\psi\in\Aut(G)$ such that
$f(x)=d\psi(x)$.
\end{enumerate}
\end{lemma}

\begin{proof}
For \textup{(i)}, setting $y=e$ gives $K(x)=f(x)$, so
$f(xy)=f(x)y$.  Setting $x=e$ gives $f(y)=f(e)y$.  The proof of
\textup{(ii)} is symmetric.

For \textup{(iii)}, put $d=f(e)$.  Setting $y=e$ and $x=e$ gives
$K(x)=f(x)d=df(x)$.  Hence $f(x)d=df(x)$ for every $x$, and surjectivity of
$f$ implies $d\in Z(G)$.  Also $f(x)f(y)=f(xy)d$.  Therefore
$\psi(x)=d^{-1}f(x)$ is a bijective homomorphism, hence an automorphism.
\end{proof}

For a nonempty subset $S\subseteq\{1,\ldots,r\}$, let $H_S$ be the
height-two tree whose $j$th child is an $r$-corolla when $j\in S$ and a leaf
otherwise.

\begin{proposition}[Height-two classification]\label{prop:height-two}
Let $f\colon G\to G$ be a bijection and put
$\omega(x_1,\ldots,x_r)=f(x_1\cdots x_r)$.  The term $\omega_{H_S}$ factors
through the ordered group product precisely as follows.
\begin{enumerate}[label=\textup{(\alph*)},leftmargin=*]
\item If $S=\{1\}$, then $f(z)=az$ for some $a\in G$.
\item If $S=\{r\}$, then $f(z)=zb$ for some $b\in G$.
\item If $S=\{1,\ldots,r\}$, then
$f(z)=d\psi(z)$ for some $d\in Z(G)$ and $\psi\in\Aut(G)$.
\item For every other nonempty $S$, one has $f(z)=cz$ for some
$c\in Z(G)$.
\end{enumerate}
The parameters in each displayed form are unique.
\end{proposition}

\begin{proof}
For each child of the root, let $u_j$ be the ordered product of the leaves in
that child subtree.  The variables $u_1,\ldots,u_r$ range independently over
$G$.  If $j\in S$, the value entering the root is $f(u_j)$; otherwise it is
$u_j$.  Since the final application of $f$ is bijective, the required
dependence on $u_1\cdots u_r$ is equivalent to
\begin{equation}\label{eq:word-pexider}
  h_1(u_1)\cdots h_r(u_r)=K(u_1\cdots u_r),
\end{equation}
where $h_j=f$ for $j\in S$ and $h_j=\id$ otherwise.

Fix all variables except two adjacent ones.  Write a $1$ in position $j$ when $j\in S$ and a $0$ otherwise.  An adjacent
pattern $10$, $01$, or $11$ in this $0$--$1$ word gives, respectively, one of
the three
equations in \cref{lem:pexider}.  The exceptional sets $\{1\}$, $\{r\}$,
and $\{1,\ldots,r\}$ therefore give the first three cases.  Every other
nonempty indicator word contains two incompatible nontrivial adjacent
patterns.  The corresponding families intersect only in the central
translations, which gives \textup{(d)}.  Direct substitution proves the
converse implications.
\end{proof}

For example, when $r=3$ and $S=\{1,3\}$, the root receives
$f(u_1),u_2,f(u_3)$.  The adjacent patterns $10$ and $01$ force $f$ to be
both a left and a right translation, so $f(z)=cz$ with $c\in Z(G)$.

\begin{lemma}[Height-two fringe]\label{lem:height-two-fringe}
Every full ordered $r$-ary tree with at least two internal vertices contains
some $H_S$, with $S\ne\varnothing$, as a fringe subtree.
\end{lemma}

\begin{proof}
Choose a deepest internal vertex $v$ and let $w$ be its parent.  Every internal
child of $w$ has only leaves as children.  Taking $S$ to be the set of
positions of these internal children gives the required fringe.
\end{proof}

Combining \cref{cor:one-tree-reduction,prop:height-two,lem:height-two-fringe}
identifies the only three possible algebraic forms of $\omega$.  To finish
the fixed-tree problem, we still have to determine which parameters in those
forms are allowed by the chosen tree.  That requires explicit formulas for
the tree iterates.
\section{Tree formulas: fixed tree and fixed operation}\label{sec:transport}

We compute the iterate of each of the three forms on an arbitrary tree $T$.
The same formulas answer the first two problems in opposite directions.  With
$T$ fixed, they give the exact restrictions on the parameters $a,b,\psi$ and
complete the classification of $\omega$.  After that, with $\omega$ fixed,
they tell us exactly which trees admit a bijection $F_T$ such that
$\omega_T(x_1,\ldots,x_n)=F_T(x_1\cdots x_n)$.

Let $i$ be a leaf of a full ordered $r$-ary tree $T$.  Follow the path from
the root to $i$.  Whenever the path enters one child of an internal vertex,
there may be sibling subtrees on its left and on its right.  Define
$\rho_T(i)$ to be the total number of internal vertices in all right sibling
subtrees encountered along the path, and define $\lambda_T(i)$ in the same
way using left sibling subtrees.  Let $\delta_T(i)$ be the depth of the leaf.
Thus the rightmost leaf has $\rho_T=0$, and the leftmost leaf has
$\lambda_T=0$.  For the three binary trees in \cref{fig:binary-height-two},
the vectors are
\[
\begin{array}{c|c|c|c}
T & (\rho_T(i)) & (\lambda_T(i)) & (\delta_T(i))\\ \hline
\text{left vine} & (0,0,0) & (0,0,1) & (2,2,1)\\
\text{right vine} & (1,0,0) & (0,0,0) & (1,2,2)\\
\text{balanced} & (1,1,0,0) & (0,0,1,1) & (2,2,2,2).
\end{array}
\]
These small cases already show why the three statistics distinguish the three
families.

\subsection{Translations}

Fix $a\in G$ and write
$\omega_a(x_1,\ldots,x_r)=a x_1\cdots x_r$.

\begin{proposition}[Left translation]\label{prop:left-formula}
For every full ordered $r$-ary tree $T$ with leaves $1,\ldots,n$,
\begin{equation}\label{eq:left-formula}
  (\omega_a)_T(x_1,\ldots,x_n)
  =a^{N(T)}\prod_{i=1}^n a^{-\rho_T(i)}x_i a^{\rho_T(i)}.
\end{equation}
Hence there is a bijection $F_T\colon G\to G$ such that
$(\omega_a)_T(x_1,\ldots,x_n)=F_T(x_1\cdots x_n)$ if and only if
\begin{equation}\label{eq:left-criterion}
  a^{\rho_T(i)}\in Z(G)\qquad(1\le i\le n).
\end{equation}
When this holds, $F_T(z)=a^{N(T)}z$.
\end{proposition}

\begin{proof}
Induct on $N(T)$.  The leaf case is immediate.  Let the root subtrees be
$T_1,\ldots,T_r$, put $N_j=N(T_j)$, and write the contribution of child $j$
as $a^{N_j}P_j$ by induction.  Then
\[
  (\omega_a)_T=a(a^{N_1}P_1)\cdots(a^{N_r}P_r).
\]
For $1\le j\le r$ let $R_j=\sum_{q>j}N_q$.  Moving the powers of $a$ to the
left gives
\[
  a(a^{N_1}P_1)\cdots(a^{N_r}P_r)
  =a^{1+\sum_jN_j}\prod_{j=1}^r a^{-R_j}P_j a^{R_j}.
\]
Every leaf in $T_j$ therefore acquires the additional exponent $R_j$, which
is precisely the contribution of the right sibling subtrees at the root.
This proves \eqref{eq:left-formula}.

If \eqref{eq:left-criterion} holds, every conjugation in
\eqref{eq:left-formula} is trivial.  Conversely, suppose the left-hand side
depends only on $x_1\cdots x_n$.  Set all variables except $x_i=x$ equal to
$e$ and compare with the same specialization at the rightmost leaf, whose
$\rho$-value is $0$.  We obtain
$a^{-\rho_T(i)}xa^{\rho_T(i)}=x$ for all $x\in G$, which is equivalent to
\eqref{eq:left-criterion}.
\end{proof}

The corresponding formula for a right translation follows by reflection.

\begin{proposition}[Right translation]\label{prop:right-formula}
Let $b\in G$ and put
$\omega^b(x_1,\ldots,x_r)=x_1\cdots x_r b$.  Then
\begin{equation}\label{eq:right-formula}
  (\omega^b)_T(x_1,\ldots,x_n)
  =\left(\prod_{i=1}^n b^{\lambda_T(i)}x_i b^{-\lambda_T(i)}\right)b^{N(T)}.
\end{equation}
There is a bijection $F_T\colon G\to G$ with
$(\omega^b)_T(x_1,\ldots,x_n)=F_T(x_1\cdots x_n)$ if and only if
\begin{equation}\label{eq:right-criterion}
  b^{\lambda_T(i)}\in Z(G)\qquad(1\le i\le n),
\end{equation}
and then $F_T(z)=zb^{N(T)}$.
\end{proposition}

\subsection{The automorphism form}

Let $d\in Z(G)$ and $\psi\in\Aut(G)$, and put
$\omega_{d,\psi}(x_1,\ldots,x_r)=d\,\psi(x_1\cdots x_r)$.  For an internal
vertex $v$, let $\delta_T(v)$ be its depth, and let $\Int(T)$ denote the set
of internal vertices.  Set
\begin{equation}\label{eq:central-factor}
  C_T(d,\psi)=\prod_{v\in\Int(T)}\psi^{\delta_T(v)}(d).
\end{equation}
The factors commute because they lie in $Z(G)$.

\begin{proposition}[The form $d\,\psi(x_1\cdots x_r)$]
\label{prop:affine-formula}
For every full ordered $r$-ary tree $T$,
\begin{equation}\label{eq:affine-formula}
  (\omega_{d,\psi})_T(x_1,\ldots,x_n)
  =C_T(d,\psi)\prod_{i=1}^n\psi^{\delta_T(i)}(x_i).
\end{equation}
There is a bijection $F_T\colon G\to G$ with
$(\omega_{d,\psi})_T(x_1,\ldots,x_n)=F_T(x_1\cdots x_n)$ if and only if
\begin{equation}\label{eq:depth-criterion}
  \psi^{\delta_T(1)}=\cdots=\psi^{\delta_T(n)}
  \qquad\text{in }\Aut(G).
\end{equation}
If the common automorphism is $\varphi$, then
$F_T(z)=C_T(d,\psi)\varphi(z)$.
\end{proposition}

\begin{proof}
Induct on $N(T)$.  Applying $\omega_{d,\psi}$ at the root contributes one
factor $d$, applies one additional $\psi$ to every contribution from a child,
and increases the depth of every child vertex and leaf by one.  This gives
\eqref{eq:affine-formula}.

Condition \eqref{eq:depth-criterion} is sufficient because an automorphism
preserves the ordered product.  Conversely, if the expression in
\eqref{eq:affine-formula} depends only on $x_1\cdots x_n$, set every variable
except $x_i=x$ equal to $e$.  The resulting unary map is
$x\mapsto C_T(d,\psi)\psi^{\delta_T(i)}(x)$ and must be independent of $i$.
\end{proof}

\subsection{Classification from one tree}

For a full ordered $r$-ary tree $T$ with leaves $1,\ldots,n$, put
\begin{align}
  g_R(T)&=\gcd\bigl(\rho_T(1),\ldots,\rho_T(n)\bigr),\label{eq:gR}\\
  g_L(T)&=\gcd\bigl(\lambda_T(1),\ldots,\lambda_T(n)\bigr),\label{eq:gL}\\
  g_D(T)&=\gcd\bigl(\delta_T(2)-\delta_T(1),\ldots,
                     \delta_T(n)-\delta_T(1)\bigr).\label{eq:gD}
\end{align}
The gcd of an all-zero list is $0$.  B\'ezout's identity turns
\eqref{eq:left-criterion}, \eqref{eq:right-criterion}, and
\eqref{eq:depth-criterion} into the three group conditions below; when the
gcd is $0$, the corresponding condition is automatic.

\begin{theorem}[One-tree classification]\label{thm:main-classification}
Let $T$ be a full ordered $r$-ary tree with $N(T)\ge2$, and let
$\omega\colon G^r\to G$ be arbitrary.  There is a bijection
$F_T\colon G\to G$ such that
\[
  \omega_T(x_1,\ldots,x_n)=F_T(x_1\cdots x_n)
\]
if and only if at least one of the following holds.
\begin{enumerate}[label=\textup{(\Alph*)},leftmargin=3.2em]
\item $\omega(x_1,\ldots,x_r)=a x_1\cdots x_r$ for some $a\in G$, and
$(aZ(G))^{g_R(T)}=Z(G)$.
\item $\omega(x_1,\ldots,x_r)=x_1\cdots x_r b$ for some $b\in G$, and
$(bZ(G))^{g_L(T)}=Z(G)$.
\item $\omega(x_1,\ldots,x_r)=d\,\psi(x_1\cdots x_r)$ for some
$d\in Z(G)$ and $\psi\in\Aut(G)$, and $\psi^{g_D(T)}=\id_G$.
\end{enumerate}
The parameters are unique within each family.  Two different families can
represent the same operation only when
\begin{equation}\label{eq:central-translations}
  \omega(x_1,\ldots,x_r)=c x_1\cdots x_r
  \qquad(c\in Z(G)),
\end{equation}
and for every such central translation and every full ordered $r$-ary tree
$T$, there is a bijection $F_T$ with
$\omega_T(x_1,\ldots,x_n)=F_T(x_1\cdots x_n)$.
\end{theorem}

\begin{proof}
Suppose such a bijection $F_T$ exists.  By \cref{cor:one-tree-reduction},
$\omega=f\circ m_r$ for a unique bijection $f$.  By
\cref{lem:height-two-fringe}, $T$ contains a fringe $H_S$ with
$S\ne\varnothing$, and \cref{thm:relative-one-tree} shows that the analogous
identity holds on that fringe.  The height-two classification
\cref{prop:height-two} therefore places $f$ in one of the three families.
The conditions on $a$, $b$, or $\psi$ follow from
\cref{prop:left-formula,prop:right-formula,prop:affine-formula} and the
B\'ezout observation above.  The converse follows from the same formulas.

Uniqueness inside each family follows by evaluating at the identity and, in
the third family, from $\psi(x)=d^{-1}\omega(x,e,\ldots,e)$.  If a left and
a right translation coincide, then $az=zb$ for every $z$, so
$a=b\in Z(G)$.  If a translation equals $d\psi(z)$ with $d$ central, then
evaluation at $e$ identifies the constants and forces $\psi=\id$.
\end{proof}

This completes the first problem: for the chosen tree $T$, we have classified every possible $\omega$ and the allowed parameters.  We now read the same formulas in the reverse direction.  Fix one operation of the forms in \eqref{eq:intro-three-families}, let $T$ vary, and ask for which trees there exists a bijection $F_T$ with $\omega_T(x_1,\ldots,x_n)=F_T(x_1\cdots x_n)$.

\begin{corollary}[Which trees satisfy the product equation]\label{cor:torsion-spectrum}
Fix one of the operations in \eqref{eq:intro-three-families}, and let $T$ be
an arbitrary full ordered $r$-ary tree.  We ask whether there exists a
bijection $F_T\colon G\to G$ such that
\[
  \omega_T(x_1,\ldots,x_n)=F_T(x_1\cdots x_n).
\]
The answer is as follows.
\begin{enumerate}[label=\textup{(\roman*)}]
\item If $\omega(x_1,\ldots,x_r)=c x_1\cdots x_r$ with $c\in Z(G)$, every
full ordered $r$-ary tree $T$ works.
\item If $\omega(x_1,\ldots,x_r)=a x_1\cdots x_r$ with $a\notin Z(G)$,
let $k$ be the order of $aZ(G)$ in $G/Z(G)$.  If $k<\infty$, such a
bijection $F_T$ exists exactly when $k$ divides every $\rho_T(i)$; if
$k=\infty$, it exists exactly when every $\rho_T(i)$ is $0$.
\item If $\omega(x_1,\ldots,x_r)=x_1\cdots x_r b$ with $b\notin Z(G)$,
let $k$ be the order of $bZ(G)$ in $G/Z(G)$.  If $k<\infty$, such a
bijection $F_T$ exists exactly when $k$ divides every $\lambda_T(i)$; if
$k=\infty$, it exists exactly when every $\lambda_T(i)$ is $0$.
\item If $\omega(x_1,\ldots,x_r)=d\psi(x_1\cdots x_r)$ with
$\psi\ne\id$, let $h$ be the order of $\psi$ in $\Aut(G)$.  If
$h<\infty$, such a bijection exists exactly when all leaf depths are
congruent modulo $h$; if $h=\infty$, it exists exactly when all leaf depths
are equal.
\end{enumerate}
\end{corollary}

We now turn these numerical conditions into descriptions of the trees.  For
$s\ge0$, let $V^R_{r,s}$ be the right $r$-ary vine of length $s$.  It has
$s$ internal vertices in one chain: except at the bottom of the chain, the
rightmost child of each internal vertex is the next internal vertex, and every
child not on the chain is a leaf.  We take $V^R_{r,0}$ to be a single leaf.
Let $V^L_{r,s}$ be the mirror image, with the chain running through leftmost
children.

\begin{proposition}[Trees for translations]
\label{prop:factorizing-vine-blocks}
Let $k\ge1$.  A full ordered $r$-ary tree $T$ satisfies
$k\mid\lambda_T(i)$ for every leaf $i$ if and only if, writing
$s\in\{0,\ldots,k-1\}$ for $N(T)$ modulo $k$, the tree $T$ can be obtained
from $V^R_{r,s}$ by repeatedly replacing a leaf with $V^R_{r,k}$.

If every $\lambda_T(i)$ is $0$, then $T=V^R_{r,N(T)}$.  The reflected
statements hold for the numbers $\rho_T(i)$ and the left vines
$V^L_{r,s}$.
\end{proposition}

\begin{proof}
Every leaf of a right vine has $\lambda=0$.  Replacing a leaf by
$V^R_{r,k}$ leaves the $\lambda$-values of earlier leaves unchanged, gives
each new leaf the old value of the replaced leaf, and increases the value of
every later leaf by $k$.  Hence the divisibility condition is preserved.

Conversely, induct on $N(T)$.  Write the root subtrees as
$T_0,\ldots,T_{r-1}$, let $N_j=N(T_j)$, and put
$S_j=N_0+\cdots+N_{j-1}$.  The leftmost leaf of $T_j$ has global value
$S_j$, so $k\mid S_j$ for every $j$.  Successive differences give
$k\mid N_0,\ldots,N_{r-2}$.  Every leaf of $T_j$ has global value
$S_j+\lambda_{T_j}$, so each $T_j$ satisfies the same divisibility condition.

By induction, the first $r-1$ subtrees reduce, after contracting
$V^R_{r,k}$ copies, to leaves.  The last subtree reduces to
$V^R_{r,t}$, where $t\equiv N_{r-1}\pmod k$ and $0\le t<k$.  The reduced
root is therefore $V^R_{r,t+1}$ when $t<k-1$; when $t=k-1$ it is
$V^R_{r,k}$ and contracts once more to $V^R_{r,0}$.  In either case the
remaining length is $N(T)$ modulo $k$.

If all $\lambda$-values are $0$, the same argument gives
$N_0=\cdots=N_{r-2}=0$ and reduces the last subtree recursively to a right
vine.  Thus $T=V^R_{r,N(T)}$.  Reflection proves the statements for
$\rho_T$.
\end{proof}

\begin{example}[Why the quotient by the center appears]\label{ex:q8-translation}
Let $G=Q_8=\{\pm1,\pm i,\pm j,\pm k\}$ and
$\omega(x,y)=xyi$.  The element $i$ has order $4$, but $Z(Q_8)=\{\pm1\}$,
so $iZ(Q_8)$ has order $2$.  The relevant block is therefore the binary right
vine with two internal vertices, and on this block
\[
  \omega(x,\omega(y,z))=xyzi^2=-xyz.
\]
Replacing any chosen leaf by another right $2$-vine preserves the condition in
\cref{prop:factorizing-vine-blocks}; the resulting tree need not itself be a
vine.  For instance, inserting a second $2$-vine at the leftmost leaf gives
\[
  \omega\bigl(\omega(x_1,\omega(x_2,x_3)),\omega(x_4,x_5)\bigr)
  =x_1x_2x_3x_4x_5.
\]
Thus the period is the order of $i$ modulo the center, not the order of $i$ in
$Q_8$.  By contrast, if $G$ is a centerless free group and $b\ne e$, then
$bZ(G)$ has infinite order.  For $\omega(x,y)=xyb$, a tree $T$ admits a
bijection $F_T$ with $\omega_T=F_T(x_1\cdots x_n)$ only when $T$ itself is
a right vine.
\end{example}

For the third family we first count vertices by depth modulo $h$.  Fix $h\ge1$.  For
$s\in\Z/h\Z$, let $I_s(T)$ be the number of internal vertices whose depths
are congruent to $s$ modulo $h$, and let $L_s(T)$ be the number of leaves with
that depth residue.

\begin{lemma}[Counting depth residues]\label{lem:residue-counts}
The vector $(L_0,\ldots,L_{h-1})$ determines
$(I_0,\ldots,I_{h-1})$.  If all leaves have one common depth residue
$c\in\{0,\ldots,h-1\}$ and $n=L(T)$, then
\begin{equation}\label{eq:leaf-count-depth-residue}
  n\equiv r^c\pmod{r^h-1}.
\end{equation}
For a fixed $n$, at most one residue $c$ can occur.
\end{lemma}

\begin{proof}
If $h=1$, the leaf-count identity gives $(r-1)I_0=L_0-1$, and the remaining
claims are immediate.  Assume $h\ge2$.  Every non-root vertex has a unique
internal parent, so counting children by depth residue gives
\begin{align}
  I_s+L_s&=rI_{s-1}\qquad(1\le s\le h-1),
     \label{eq:residue-recurrence}\\
  I_0+L_0&=1+rI_{h-1}.
     \label{eq:residue-root}
\end{align}
Eliminating $I_1,\ldots,I_{h-1}$ yields
\begin{equation}\label{eq:I0-formula}
  (r^h-1)I_0
   =L_0-1+\sum_{j=1}^{h-1}r^{h-j}L_j.
\end{equation}
Thus the leaf counts determine $I_0$ and then every $I_s$.

If all leaves have residue $c$, \eqref{eq:I0-formula} gives
$n\equiv1\pmod{r^h-1}$ when $c=0$ and
$r^{h-c}n\equiv1\pmod{r^h-1}$ when $c>0$; both are equivalent to
\eqref{eq:leaf-count-depth-residue}.  Finally, if
$r^c\equiv r^{c'}\pmod{r^h-1}$ with $0\le c<c'<h$, then cancellation is
valid because $\gcd(r,r^h-1)=1$, and it would force
$r^h-1\mid r^{c'-c}-1$, impossible since
$0<r^{c'-c}-1<r^h-1$.
\end{proof}

Write $P_{r,j}$ for the perfect full $r$-ary tree of height $j$; in
particular, $P_{r,0}$ is a leaf.

\begin{proposition}[Trees with one depth residue]
\label{prop:factorizing-perfect-blocks}
Fix $h\ge1$ and $c\in\{0,\ldots,h-1\}$.  All leaves of a full ordered
$r$-ary tree $T$ have depth congruent to $c$ modulo $h$ if and only if $T$
can be obtained from $P_{r,c}$ by repeatedly replacing a leaf with
$P_{r,h}$.  Hence every such tree has
\begin{equation}\label{eq:affine-leaf-count-support}
  L(T)=r^c+t(r^h-1)\qquad(t\ge0),
\end{equation}
and every $t\ge0$ occurs.  The numerical condition alone does not determine
the shape of $T$.
\end{proposition}

\begin{proof}
Replacing a leaf by $P_{r,h}$ adds $h$ to the depths of all new leaves, so it
preserves their depth modulo $h$ and adds $r^h-1$ leaves.

Conversely, suppose all leaf depths are congruent to $c$ modulo $h$.  If the
maximum leaf depth is $c$, then every leaf has depth $c$, so $T=P_{r,c}$.
Otherwise let $d$ be the maximum leaf depth and choose a deepest leaf.  Its
ancestor at depth $d-h$ exists.  Any leaf below that ancestor has depth
strictly larger than $d-h$, at most $d$, and congruent to $d$ modulo $h$.
The only possibility is depth $d$.  Hence that fringe subtree is
$P_{r,h}$.  Contract it to one leaf and repeat.  This eventually reduces $T$
to $P_{r,c}$ and proves the claim.
\end{proof}

\begin{example}[Parity of leaf depths]\label{ex:c3-inversion}
Let $G=C_3$ and let $\psi(x)=x^{-1}$, which has order $2$.  For
$\omega(x,y)=(xy)^{-1}$, a binary tree satisfies the identity with the ordered
product exactly when all leaf depths have the same parity.  The perfect tree of
height $2$ works.  If one of its leaves is replaced by another perfect tree of
height $2$, the new tree is no longer perfect---some leaves have depth $2$ and
the new ones have depth $4$---but it still works.  In contrast, the binary
right vine with three leaves has depths $1,2,2$ and does not work.  Even leaf depths occur at leaf counts $n\equiv1\pmod3$, while odd leaf
depths occur at $n\equiv2\pmod3$.
\end{example}

The same description also determines the possible leaf counts in the third
family.  If $\psi$ has finite order $h$, then among trees with at least two
internal vertices there is a tree satisfying the identity at leaf count $n$
exactly when $n\ge2r-1$ and
$n\equiv r^c\pmod{r^h-1}$ for some $0\le c<h$.  If $\psi$ has infinite order, all leaf depths must be equal, so the whole
tree is perfect; the possible leaf counts are $r^j$ for $j\ge2$.

Moreover, whenever such a bijection $F_T$ exists, the number of leaves already
determines it.

\begin{corollary}[Equal leaf counts give the same term]
\label{cor:same-size-coherence}
Let $\omega$ be one of the three operations in
\eqref{eq:intro-three-families}.  Let $T$ and $U$ have the same number of
leaves, and suppose there are bijections
$F_T,F_U\colon G\to G$ such that
\[
  \omega_T(x_1,\ldots,x_n)=F_T(x_1\cdots x_n),\qquad
  \omega_U(x_1,\ldots,x_n)=F_U(x_1\cdots x_n).
\]
Then $F_T=F_U$, and hence $\omega_T=\omega_U$.
\end{corollary}

\begin{proof}
For a left translation, $F_T(z)=a^{N(T)}z$; for a right translation,
$F_T(z)=zb^{N(T)}$.  Equal leaf counts imply equal numbers of internal
vertices by \eqref{eq:leaf-count}, so the maps agree.

Now let $\omega=d\psi(x_1\cdots x_r)$.  If $\psi=\id$, this is already a
central translation.  Suppose first that $h=\ord(\psi)<\infty$.  The leaves
of $T$ have one depth residue $c_T$ modulo $h$, and those of $U$ have one
residue $c_U$.  By \cref{lem:residue-counts} and the common leaf count,
$r^{c_T}\equiv r^{c_U}\pmod{r^h-1}$, so $c_T=c_U$.  The leaf-residue
vectors are therefore the same, and \cref{lem:residue-counts} gives
$I_s(T)=I_s(U)$ for every $s$.  Hence
$C_T(d,\psi)=C_U(d,\psi)$, while the common automorphism coefficient in
\eqref{eq:affine-formula} is $\psi^{c_T}$.  Thus $F_T=F_U$.

If $\psi$ has infinite order, all leaf depths in each tree must be equal.
Such a full $r$-ary tree is perfect, and its leaf count determines its height,
so $T=U$.
\end{proof}

We have now solved the reversed problem: for each fixed operation in the three families, we know exactly which individual trees make the iterate depend only on the ordered group product.  Before turning to the separate question of associativity, we give an optional structural explanation for the three algebraic forms.
\section{Why the three forms appear}\label{sec:crossed-module}

This section is not needed for the classification proof.  It gives a compact
explanation for why the three formulas in \eqref{eq:intro-three-families}
are the only ones that appear.

An affine permutation of $G$ has the form $h(x)=c\alpha(x)$ with
$c\in G$ and $\alpha\in\Aut(G)$.  Such maps form the holomorph
$\Hol(G)=G\rtimes\Aut(G)$.  The standard automorphism crossed module
$G\xrightarrow{\Ad}\Aut(G)$ equips an affine map with two automorphisms,
which we write as
\begin{equation}\label{eq:source-target}
  s(c,\alpha)=\alpha,
  \qquad
  t(c,\alpha)=\Ad_c\alpha.
\end{equation}
The only fact about this crossed module that we use is the following criterion.

\begin{proposition}[When a product of affine maps depends only on the product]
\label{prop:block-composable}
Let $h_j(x)=c_j\alpha_j(x)$ for $1\le j\le r$.  Then
$h_1(x_1)\cdots h_r(x_r)$ depends, through a bijection, only on
$x_1\cdots x_r$ if and only if
\begin{equation}\label{eq:chain-condition}
  s(h_j)=t(h_{j+1})\qquad(1\le j<r).
\end{equation}
When these equalities hold, the resulting map of $x_1\cdots x_r$ is again an
affine bijection.  In the crossed-module language it is the composite of the
arrows $h_1,\ldots,h_r$.
\end{proposition}

\begin{proof}
Put $E_j=c_{j+1}\cdots c_r$.  Moving the constants to the left gives
\[
  h_1(x_1)\cdots h_r(x_r)
  =(c_1\cdots c_r)
    \prod_{j=1}^r\bigl(\Ad_{E_j^{-1}}\alpha_j\bigr)(x_j).
\]
This expression can depend only on $x_1\cdots x_r$ exactly when the displayed
automorphisms are all equal.  Necessity follows by setting all variables but
one equal to $e$; sufficiency follows because an automorphism preserves the
ordered product.  Equality at adjacent positions is equivalent to
$\alpha_j=\Ad_{c_{j+1}}\alpha_{j+1}$, which is precisely
$s(h_j)=t(h_{j+1})$.
\end{proof}

Now write the bijection in \eqref{eq:output-isotope} as
$f(x)=a\phi(x)$.  At a deepest height-two fringe, every child contributes
either the identity map or $f$.  The nontrivial ways these maps can satisfy
\eqref{eq:chain-condition} force one of
\[
  s(f)=1,\qquad t(f)=1,\qquad s(f)=t(f).
\]
Using \eqref{eq:source-target}, these three equations say, respectively,
\[
  f(x)=ax,\qquad f(x)=xa,\qquad
  f(x)=a\phi(x)\ \text{with }a\in Z(G).
\]
These are exactly the three forms already obtained from the elementary
two-variable argument in \cref{lem:pexider}.

For any $(a,\phi)$ one has
$t(a,\phi)s(a,\phi)^{-1}=\Ad_a$.  Thus source and target have the same image
in $\Out(G)$, and they are both the identity exactly when
$\phi=\id$ and $a\in Z(G)$.  These two observations account for the later
appearance of $Z(G)$ and $\Out(G)$ without introducing any further
crossed-module machinery.
\section{How much associativity remains?}\label{sec:reassociation}

We now turn to a separate, related question.  Fix one of the three solution forms in \eqref{eq:intro-three-families}.  Although it is built from the associative group product, it is generally not associative.  We ask how much associativity remains: for full ordered $r$-ary trees $S$ and $T$ with the same number of leaves, when does
\begin{equation}\label{eq:reassoc-direct}
  \omega_S=\omega_T
\end{equation}
hold as an identity of operations?  Equivalently, which changes of parenthesization leave the iterate unchanged?

This is not another version of the individual-tree problem from \cref{sec:transport}.  There we asked whether a single iterate could be expressed as a bijective function of the ordered group product.  Here neither $S$ nor $T$ is required to have that property.  We compare arbitrary parenthesizations of the same fixed operation and determine all equalities between them.

A \emph{caret} is one internal vertex together with its $r$ children, and
expanding a leaf means replacing that leaf by a caret.  An element of the
Higman--Thompson group $F_r$ is represented by a pair $(T,U)$ of full ordered
$r$-ary trees with the same number of leaves.  The pair can be read as a
change from the parenthesization described by $T$ to the one described by
$U$.  Simultaneously expanding the same corresponding leaf on both sides
gives another representative of the same element.  To compose $(T,U)$ with
$(U',V)$, expand $U$ and $U'$ until they are the same tree and then compose
the two changes of parenthesization.  The identity is represented by $(T,T)$,
and swapping the two trees gives the inverse.  For $r=2$ this is Thompson's
group $F$; see Brown \cite{Brown1987}.

It is useful to define this subgroup for an arbitrary operation $q$ before
computing it for the three forms in \eqref{eq:intro-three-families}.

For an $r$-ary operation $q\colon Q^r\to Q$, let $\Assoc_r(q)$ consist of the
elements represented by pairs $(T,U)$ for which the two iterated operations
are equal after a possible simultaneous expansion.  This is a subgroup of
$F_r$: the identity pair gives $q_T=q_T$, swapping a pair reverses an
equality, and composition uses transitivity of equality after the middle
trees have been expanded to agree.

When $q$ is surjective, simultaneous expansion does not create a new identity.

\begin{proposition}[Checking an identity on a representative]
\label{prop:stable-exact}
Let $q\colon Q^r\to Q$ be surjective.  If a simultaneous expansion of a tree
pair $(T,U)$ satisfies $q_T=q_U$, then the unexpanded pair already satisfies
$q_T=q_U$.  Hence, for every representative $(T,U)$ of an element of $F_r$,
\[
  (T,U)\in\Assoc_r(q)\quad\Longleftrightarrow\quad q_T=q_U.
\]
\end{proposition}

\begin{proof}
It is enough to contract one common expanded leaf.  The expanded identity is
obtained by substituting the same value $q(y_1,\ldots,y_r)$ for one variable on
both sides.  Since $q$ is surjective, this value ranges over all of $Q$, so
the unexpanded identity follows.  Repeating the contraction proves the claim.
\end{proof}

Each of the three operations in \eqref{eq:intro-three-families} is a
quasigroup, hence surjective.  Therefore, for the operations studied here, an
identity can be checked on any representative of the corresponding element of
$F_r$.

\subsection{Replacing each internal vertex by a fixed tree}

Number the children of every internal vertex by $0,1,\ldots,r-1$ from left
to right.  The \emph{address} of a leaf is the word obtained by recording
these child numbers along the path from the root.

Let $C$ be a full ordered $r$-ary tree with $m$ leaves, whose leaf addresses
are $c_0<\cdots<c_{m-1}$.  No address is a prefix of another, and every
infinite $r$-ary word begins with exactly one of them; in other words, the
addresses form a complete prefix code.  If we replace every internal vertex of an $m$-ary tree by a copy of $C$,
using the $m$ leaves of $C$ as the attachment points for its $m$ children,
tree pairs give an embedding
\begin{equation}\label{eq:code-tree-embedding}
  \iota_C\colon F_m\hookrightarrow F_r.
\end{equation}
Such embeddings, often called caret-replacement embeddings, are standard in
Thompson-group theory; see, for example, Golan--Sapir \cite{GolanSapir2017}.
The two simplest substitutions used later are shown in
\cref{fig:basic-substitutions}.

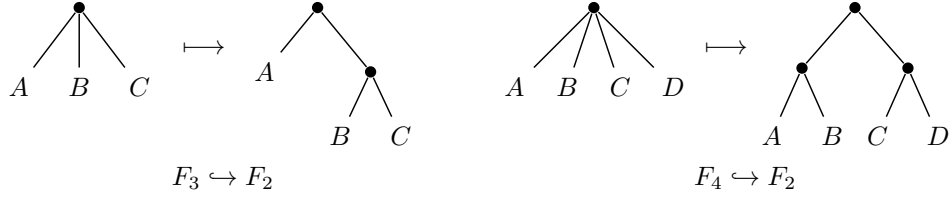
\begin{figure}[t]
\centering
\begin{tikzpicture}[line width=.55pt, font=\small]
  \begin{scope}[xshift=0cm]
    \node[circle,fill=black,inner sep=1.5pt] (T0) at (0,1.7) {};
    \node (TA) at (-.8,.65) {$A$};
    \node (TB) at (0,.65) {$B$};
    \node (TC) at (.8,.65) {$C$};
    \draw (T0)--(TA) (T0)--(TB) (T0)--(TC);
    \node at (1.65,1.15) {$\longmapsto$};
    \node[circle,fill=black,inner sep=1.5pt] (U0) at (3.15,1.7) {};
    \node (UA) at (2.45,.85) {$A$};
    \node[circle,fill=black,inner sep=1.5pt] (U1) at (3.85,.85) {};
    \node (UB) at (3.45,0) {$B$};
    \node (UC) at (4.25,0) {$C$};
    \draw (U0)--(UA) (U0)--(U1) (U1)--(UB) (U1)--(UC);
    \node at (1.9,-.55) {$F_3\hookrightarrow F_2$};
  \end{scope}
  \begin{scope}[xshift=6.8cm]
    \node[circle,fill=black,inner sep=1.5pt] (Q0) at (0,1.7) {};
    \node (QA) at (-1.05,.65) {$A$};
    \node (QB) at (-.35,.65) {$B$};
    \node (QC) at (.35,.65) {$C$};
    \node (QD) at (1.05,.65) {$D$};
    \draw (Q0)--(QA) (Q0)--(QB) (Q0)--(QC) (Q0)--(QD);
    \node at (1.75,1.15) {$\longmapsto$};
    \node[circle,fill=black,inner sep=1.5pt] (P0) at (3.45,1.7) {};
    \node[circle,fill=black,inner sep=1.5pt] (P1) at (2.75,.9) {};
    \node[circle,fill=black,inner sep=1.5pt] (P2) at (4.15,.9) {};
    \node (PA) at (2.35,0) {$A$};
    \node (PB) at (3.15,0) {$B$};
    \node (PC) at (3.75,0) {$C$};
    \node (PD) at (4.55,0) {$D$};
    \draw (P0)--(P1) (P0)--(P2) (P1)--(PA) (P1)--(PB) (P2)--(PC) (P2)--(PD);
    \node at (2.0,-.55) {$F_4\hookrightarrow F_2$};
  \end{scope}
\end{tikzpicture}
\caption{The simplest nontrivial substitutions used in the reassociation
theorem.  On the left, a ternary internal vertex is replaced by a binary right
$2$-vine; on the right, a quaternary internal vertex is replaced by the
perfect binary tree of height $2$.  These are the substitutions arising in
\cref{ex:q8-translation,ex:c3-inversion}.}
\label{fig:basic-substitutions}
\end{figure}

To recognize the image of \eqref{eq:code-tree-embedding}, take a finite
$r$-ary word $w$ and greedily remove complete codewords
$c_j$ from the left.  The part left over is either empty or a proper prefix of
a codeword.  Denote this unfinished suffix by $\operatorname{rem}_C(w)$.  For the binary
right $2$-vine the code is $\{0,10,11\}$; for example,
$1101=11\mid0\mid1$ leaves the unfinished suffix $1$.

\begin{proposition}[Which pairs come from a fixed-tree substitution]
\label{prop:code-tree-state}
An element $g\in F_r$ belongs to $\iota_C(F_m)$ if and only if it has a tree
pair representative whose corresponding leaf addresses $u_i$ and $v_i$
satisfy
\begin{equation}\label{eq:equal-code-state}
  \operatorname{rem}_C(u_i)=\operatorname{rem}_C(v_i)
  \qquad\text{for every }i.
\end{equation}
\end{proposition}

\begin{proof}
If the tree pair comes from an $m$-ary pair by substituting $C$, every leaf
word is a concatenation of complete codewords, so both remainders are empty.

Conversely, suppose \eqref{eq:equal-code-state} holds.  For each corresponding
pair of leaves, expand both leaves by the same part of $C$ needed to complete
the common unfinished suffix.  After doing this for every pair, every leaf address on both sides is a concatenation of complete codewords.  Replace each
$c_j$ by the digit $j$.  The resulting two sets of words are complete prefix
codes over an alphabet of size $m$, hence the leaf sets of two full ordered
$m$-ary trees.  The expanded pair therefore comes from an $m$-ary tree pair
by substituting $C$, which proves the claim.
\end{proof}

We will also use the two endpoint slopes in the usual piecewise-linear model
of $F_r$.  Write
\[
  \partial_r(g)
  =\bigl(\log_r g'(0^+),\log_r g'(1^-)\bigr)\in\Z^2.
\]
Let $\ell(C)$ and $\varrho(C)$ be the depths of the leftmost and rightmost
leaves of $C$.

\begin{proposition}[Endpoint slopes under tree substitution]
\label{prop:endpoint-code-tree}
\begin{equation}\label{eq:endpoint-code-tree}
  \partial_r\bigl(\iota_C(F_m)\bigr)
  =\ell(C)\Z\times\varrho(C)\Z.
\end{equation}
\end{proposition}

\begin{proof}
Each occurrence of the leftmost $m$-ary digit is replaced by
$0^{\ell(C)}$, and each occurrence of the rightmost digit by
$(r-1)^{\varrho(C)}$.  Thus the left and right slope exponents are multiplied
by $\ell(C)$ and $\varrho(C)$, respectively.  The endpoint exponents in
$F_m$ can be varied independently, so \eqref{eq:endpoint-code-tree} follows.
\end{proof}

\subsection{Translations}

Consider first
$\omega^b(x_1,\ldots,x_r)=x_1\cdots x_r b$.  If a leaf has address
$w=d_1\cdots d_s$, write
$\sigma_r(w)=d_1+\cdots+d_s$.

\begin{lemma}[Leaf index and digit sum]\label{lem:index-digit}
Let $w_i$ be the address of the $i$th leaf of a full ordered $r$-ary tree.
Then
\begin{equation}\label{eq:index-digit}
  i-1=\sigma_r(w_i)+(r-1)\lambda_T(i).
\end{equation}
In particular,
$\sigma_r(w_i)\equiv i-1\pmod{r-1}$.
\end{lemma}

\begin{proof}
At a vertex where the path to $w_i$ enters child $d$, the $d$ sibling
subtrees on the left contribute all of their leaves to the $i-1$ preceding
leaves.  Each such subtree $S$ has $L(S)=1+(r-1)N(S)$.  Summing the constant
terms over the left siblings gives $\sigma_r(w_i)$, while the internal-vertex
terms give $(r-1)\lambda_T(i)$.
\end{proof}

We also need the relation between consecutive leaf addresses.

\begin{lemma}[Consecutive leaves]\label{lem:consecutive-leaves}
If $w_i,w_{i+1}$ are consecutive leaf addresses, then uniquely
\begin{equation}\label{eq:consecutive-normal-form}
  w_i=u\,j\,(r-1)^p,
  \qquad
  w_{i+1}=u\,(j+1)0^q
\end{equation}
for some word $u$, some $0\le j\le r-2$, and $p,q\ge0$.  Hence
\begin{equation}\label{eq:successor-digit}
  \sigma_r(w_i)-\sigma_r(w_{i+1})+1=(r-1)p.
\end{equation}
\end{lemma}

\begin{proof}
Let $u$ be the address of the lowest common ancestor.  Since no leaf lies
between $w_i$ and $w_{i+1}$, the two paths next enter consecutive children
$j$ and $j+1$.  The first leaf is the rightmost leaf of the $j$th child
subtree, and the second is the leftmost leaf of the $(j+1)$st child subtree.
This gives the displayed form and then the digit-sum identity.
\end{proof}

Let $k$ be the order of $bZ(G)$ in $G/Z(G)$, allowing $k=\infty$.
For two trees $T,U$ with the same number of leaves,
\cref{prop:right-formula} gives
\begin{equation}\label{eq:lambda-pair-condition}
  (\omega^b)_T=(\omega^b)_U
  \quad\Longleftrightarrow\quad
  \lambda_T(i)\equiv\lambda_U(i)\pmod{k}
  \ \text{for every }i,
\end{equation}
where congruence modulo $\infty$ means equality.  For finite $k$,
\cref{lem:index-digit} rewrites this as
\begin{equation}\label{eq:digit-pair-condition}
  \sigma_r(w_i(T))\equiv\sigma_r(w_i(U))
  \pmod{k(r-1)}
  \qquad\text{for every }i.
\end{equation}

Assume now that $k<\infty$.  The right vine $V^R_{r,k}$ has leaf addresses
\begin{equation}\label{eq:right-vine-code}
  (r-1)^t j
  \quad(0\le t<k,\ 0\le j<r-1),
  \qquad
  (r-1)^k.
\end{equation}
There are $m=1+k(r-1)$ leaves, and their digit sums, in left-to-right order,
are $0,1,\ldots,m-1$.

\begin{lemma}[Digit congruence and the vine code]
\label{lem:color-vine-state}
If two trees with the same number of leaves satisfy
\eqref{eq:digit-pair-condition}, then corresponding leaf addresses leave the
same unfinished suffix when parsed with the prefix code of $V^R_{r,k}$.
\end{lemma}

\begin{proof}
For this code, the unfinished suffix is $(r-1)^s$ with $0\le s<k$; the
integer $s$ is the length of the final run of the digit $r-1$, reduced modulo
$k$.  For a nonfinal leaf, let $p_T$ and $p_U$ be the exponents in
\eqref{eq:consecutive-normal-form}.  Applying
\eqref{eq:successor-digit} in both trees and using
\eqref{eq:digit-pair-condition} at leaves $i$ and $i+1$ gives
$(r-1)p_T\equiv(r-1)p_U\pmod{k(r-1)}$, hence
$p_T\equiv p_U\pmod k$.  The final leaf has address $(r-1)^d$, and its
digit-sum congruence gives the same conclusion directly.
\end{proof}

We also need the converse direction.  Label the vine codewords by
$0,\ldots,m-1$ in left-to-right order.  The codeword with label $j$ has digit
sum $j$.  Hence, if an $m$-ary word $a_1\cdots a_s$ is encoded by
concatenating these codewords, the resulting $r$-ary digit sum is congruent to
$a_1+\cdots+a_s$ modulo $m-1$.  Applying \cref{lem:index-digit} with arity
$m$ shows that, for the $i$th leaf of any full ordered $m$-ary tree,
$a_1+\cdots+a_s\equiv i-1\pmod{m-1}$.  Thus corresponding branches of any
$m$-ary tree pair encode to $r$-ary words with equal digit sums modulo
$m-1=k(r-1)$.

\begin{theorem}[Tree identities for translations]\label{thm:translation-reassoc}
Let $\omega^b(x_1,\ldots,x_r)=x_1\cdots x_r b$, and let $k$ be the order of
$bZ(G)$ in $G/Z(G)$.  If $k<\infty$, then
\begin{equation}\label{eq:right-vine-group}
  \Assoc_r(\omega^b)
  =\iota_{V^R_{r,k}}\bigl(F_{1+k(r-1)}\bigr)
  \cong F_{1+k(r-1)}.
\end{equation}
If $k=\infty$, then $\Assoc_r(\omega^b)=1$.

For $\omega_a(x_1,\ldots,x_r)=a x_1\cdots x_r$, if $k$ is the finite order
of $aZ(G)$, then
\[
  \Assoc_r(\omega_a)=\iota_{V^L_{r,k}}(F_{1+k(r-1)})
  \cong F_{1+k(r-1)}.
\]
If $aZ(G)$ has infinite order, then $\Assoc_r(\omega_a)=1$.
\end{theorem}

\begin{proof}
Suppose first that $k<\infty$ and put $m=1+k(r-1)$.  If a tree pair is an
identity for $\omega^b$, then \eqref{eq:digit-pair-condition} holds.  By
\cref{lem:color-vine-state}, corresponding leaf addresses leave the same unfinished
suffix for the right-vine code.  Proposition \ref{prop:code-tree-state}
therefore places the element in $\iota_{V^R_{r,k}}(F_m)$.

Conversely, take an $m$-ary tree pair and replace every internal vertex by $V^R_{r,k}$.  The preceding encoding argument shows that corresponding $r$-ary
leaf addresses have equal digit sums modulo $m-1=k(r-1)$.  Hence
\eqref{eq:lambda-pair-condition} holds and the substituted pair is an identity
for $\omega^b$.

If $k=\infty$, equality of terms gives
$\lambda_T(i)=\lambda_U(i)$ for every $i$.  By
\cref{lem:index-digit}, the ordered digit-sum sequences of $T$ and $U$ are
equal.  Equation \eqref{eq:successor-digit} then determines the exact final
run of $r-1$ for every nonfinal leaf, and the digit sum of the final leaf
determines its address $(r-1)^d$.  Working backwards through
\eqref{eq:consecutive-normal-form} reconstructs all leaf addresses.  Thus
$T=U$, so the represented element of $F_r$ is the identity.  Reflection gives
the left-handed statement.
\end{proof}

For $r=2$, this recovers the generalized Jones subgroups of Golan--Sapir
\cite{GolanSapir2017}.  For arbitrary $r$, the case $k=2$ uses the two-caret
right vine and is exactly the subgroup $H_r\cong F_{2r-1}$ of
Aiello--Nagnibeda \cite[Sec.~5, Fig.~6]{AielloNagnibeda2022}.  In particular,
\cref{ex:q8-translation} has $\Assoc_2(\omega)\cong F_3$, with the embedding
shown in the left panel of \cref{fig:basic-substitutions}.

\subsection{Automorphism form}

Let $\omega_{d,\psi}(x_1,\ldots,x_r)=d\psi(x_1\cdots x_r)$ with
$d\in Z(G)$, and let $h$ be the order of $\psi$ in $\Aut(G)$.  If $T,U$ have
the same number of leaves, \cref{prop:affine-formula} shows that equality of
the variable coefficients is equivalent to
\begin{equation}\label{eq:depth-pair-congruence}
  \delta_T(i)\equiv\delta_U(i)\pmod h
  \qquad\text{for every }i.
\end{equation}

Assume first that $h<\infty$.  The tree $P_{r,h}$ has $r^h$ leaves, so it can replace each internal
vertex of an $r^h$-ary tree.

\begin{proposition}[Which pairs come from perfect-tree substitution]
\label{prop:perfect-block-depth}
An element $g\in F_r$ belongs to
$\iota_{P_{r,h}}(F_{r^h})$ if and only if it has a tree-pair representative
$(T,U)$ with
\[
  \delta_T(i)\equiv\delta_U(i)\pmod h
  \qquad\text{for every corresponding leaf }i.
\]
\end{proposition}

\begin{proof}
A pair obtained by substituting $P_{r,h}$ has branch lengths divisible by
$h$, so the congruences hold and remain true under simultaneous expansion.

Conversely, suppose the congruences hold.  For each corresponding leaf pair,
choose $t_i\in\{0,\ldots,h-1\}$ so that both depths become divisible by $h$
after adding $t_i$, and expand both leaves by the same perfect tree of height
$t_i$.  Every branch length in the expanded pair is now divisible by $h$.
Group each address into consecutive blocks of $h$ digits and treat each block
as one symbol from an alphabet of size $r^h$.  The two leaf sets decode to
full ordered $r^h$-ary trees.  Hence the expanded pair comes from an
$r^h$-ary tree pair by substituting $P_{r,h}$.
\end{proof}

It remains to check the central factor $C_T(d,\psi)$.  If corresponding leaf
depths agree modulo $h$, then $L_s(T)=L_s(U)$ for every depth residue $s$.
Lemma \ref{lem:residue-counts} gives $I_s(T)=I_s(U)$, and therefore
\begin{equation}\label{eq:central-factor-residue}
  C_T(d,\psi)
  =\prod_{s=0}^{h-1}\psi^s(d)^{I_s(T)}
  =\prod_{s=0}^{h-1}\psi^s(d)^{I_s(U)}
  =C_U(d,\psi).
\end{equation}

\begin{theorem}[Tree identities for the automorphism form]
\label{thm:affine-reassoc}
Let $\omega_{d,\psi}(x_1,\ldots,x_r)=d\psi(x_1\cdots x_r)$ with
$d\in Z(G)$.  If $h=\ord(\psi)<\infty$, then
\begin{equation}\label{eq:perfect-block-group}
  \Assoc_r(\omega_{d,\psi})
  =\iota_{P_{r,h}}(F_{r^h})
  \cong F_{r^h}.
\end{equation}
If $\psi$ has infinite order, then $\Assoc_r(\omega_{d,\psi})=1$.
\end{theorem}

For the operation $\omega(x,y)=(xy)^{-1}$ in \cref{ex:c3-inversion},
$h=2$, so $\Assoc_2(\omega)$ is the copy of $F_4$ shown in the right panel
of \cref{fig:basic-substitutions}.

\begin{proof}
For finite $h$, \eqref{eq:depth-pair-congruence} and
\cref{prop:perfect-block-depth} identify the image of the substitution embedding, while
\eqref{eq:central-factor-residue} shows that the central factors agree.

If $\psi$ has infinite order, equality of the variable coefficients forces
$\delta_T(i)=\delta_U(i)$ for every corresponding leaf.  In the interval
model of $F_r$, the slope on the corresponding branch interval is then
$r^{\delta_T(i)-\delta_U(i)}=1$.  The resulting increasing piecewise-linear
homeomorphism fixes $0$ and has slope $1$ everywhere, so it is the identity.
\end{proof}

\subsection{Reassociation groups}

\begin{theorem}[Reassociation classification]\label{thm:main-reassoc}
Let $\omega\colon G^r\to G$ be one of the three forms in \eqref{eq:intro-three-families}.  Then:
\begin{enumerate}[label=\textup{(\roman*)},leftmargin=*]
\item if $\omega(x_1,\ldots,x_r)=c x_1\cdots x_r$ with $c\in Z(G)$, then
$\Assoc_r(\omega)=F_r$;
\item if $\omega(x_1,\ldots,x_r)=a x_1\cdots x_r$ is noncentral and the
order of $aZ(G)$ is $k<\infty$, then
$\Assoc_r(\omega)=\iota_{V^L_{r,k}}(F_{1+k(r-1)})$; if the order is infinite,
$\Assoc_r(\omega)=1$;
\item if $\omega(x_1,\ldots,x_r)=x_1\cdots x_r b$ is noncentral and the
order of $bZ(G)$ is $k<\infty$, then
$\Assoc_r(\omega)=\iota_{V^R_{r,k}}(F_{1+k(r-1)})$; if the order is infinite,
$\Assoc_r(\omega)=1$;
\item if $\omega(x_1,\ldots,x_r)=d\psi(x_1\cdots x_r)$ with
$d\in Z(G)$ and $h=\ord(\psi)<\infty$, then
$\Assoc_r(\omega)=\iota_{P_{r,h}}(F_{r^h})$; if $\psi$ has infinite order,
$\Assoc_r(\omega)=1$.
\end{enumerate}
\end{theorem}

\begin{proof}
The translation and automorphism cases are
\cref{thm:translation-reassoc,thm:affine-reassoc}.  For a central translation,
every $n$-leaf tree gives $c^{N(T)}x_1\cdots x_n$, and $N(T)$ depends only on
$n$ by \eqref{eq:leaf-count}; hence every tree pair is an identity.
\end{proof}

\begin{corollary}[Which abstract groups occur]\label{cor:complete-occurrence}
Fix $r\ge2$.  Let $G$ vary over groups and let $\omega\colon G^r\to G$ vary over the three families in \eqref{eq:intro-three-families}.  Up to isomorphism, the possible groups
$\Assoc_r(\omega)$ are
\begin{equation}\label{eq:occurrence-set}
  1
  \quad\text{and}\quad
  F_m\ \text{with }m\ge r\text{ and }m\equiv1\pmod{r-1}.
\end{equation}
Every nontrivial group in this list can be obtained from a translation over a
finite centerless group.
\end{corollary}

\begin{proof}
The theorem gives $F_r$, the groups $F_{1+k(r-1)}$, and the groups $F_{r^h}$.
Since
$r^h=1+((r^h-1)/(r-1))(r-1)$, the last family is already contained in the
translation list.  Conversely, if $m=1+k(r-1)$ with $k\ge2$, take
$G=S_{k+1}$ and let $b$ be a $k$-cycle; then $Z(G)=1$ and the order of
$bZ(G)$ is $k$.  For $m=r$, take $G=S_3$ and $b=e$.

The trivial group also occurs.  Take a centerless free group and a nontrivial
$b$.  The right vine satisfies the one-tree identity because all of its
$\lambda$-values are $0$, while the order of $bZ(G)$ is infinite, so
\cref{thm:translation-reassoc} gives $\Assoc_r(\omega^b)=1$.

The groups in \eqref{eq:occurrence-set} are pairwise nonisomorphic because
$F_m^{\mathrm{ab}}\cong\Z^m$ \cite{Brown1987,Golan2026}.
\end{proof}

Thus for $r=2$ every $F_m$ with $m\ge2$ occurs; for $r=3$ the list is
$F_3,F_5,F_7,\ldots$; and for $r=4$ it is $F_4,F_7,F_{10},\ldots$.

The abstract group does not always determine how it sits inside $F_r$.

\begin{corollary}[Endpoint slopes]\label{cor:endpoint-fingerprint}
In the finite noncentral cases of \cref{thm:main-reassoc},
\begin{align*}
  \partial_r(\Assoc_r(a x_1\cdots x_r))&=k\Z\times\Z,\\
  \partial_r(\Assoc_r(x_1\cdots x_r b))&=\Z\times k\Z,\\
  \partial_r(\Assoc_r(d\psi(x_1\cdots x_r)))&=h\Z\times h\Z,
\end{align*}
where $k$ is the order of $aZ(G)$ or $bZ(G)$, as appropriate, and $h$ is
the order of $\psi$ in $\Aut(G)$.  Thus these endpoint slopes distinguish the
three cases and recover the finite order.
\end{corollary}

\begin{proof}
The left and right vines have extreme leaf-depth pairs $(k,1)$ and $(1,k)$,
while $P_{r,h}$ has extreme depths $(h,h)$.  Apply
\cref{prop:endpoint-code-tree}.
\end{proof}

\begin{remark}[The same abstract group can occur in different ways]
\label{rem:repunit}
For a concrete case, take $r=2$.  A right translation over $S_4$ with
$b=(123)$ has $k=3$, so its reassociation subgroup is a copy of $F_4$ with
endpoint-slope image $\Z\times3\Z$.  The operation
$\omega(x,y)=(xy)^{-1}$ on $C_3$ from \cref{ex:c3-inversion} also has
reassociation subgroup $F_4$, but its endpoint-slope image is
$2\Z\times2\Z$.  Hence these two copies of $F_4$ in $F_2$ are not conjugate.

More generally, if $k=(r^h-1)/(r-1)$, then the left-vine, right-vine, and
perfect-tree subgroups are all isomorphic to $F_{r^h}$.  Their endpoint-slope
images are $k\Z\times\Z$, $\Z\times k\Z$, and $h\Z\times h\Z$, respectively,
and are distinct for $h\ge2$.  This statement concerns only the subgroups
arising here; we do not claim that endpoint slopes classify arbitrary
embeddings $F_m\hookrightarrow F_r$.
\end{remark}
\section{Multiplication groups and perfect-tree roots}\label{sec:consequences}

We finish with two consequences of the form
$\omega(x_1,\ldots,x_r)=a\phi(x_1\cdots x_r)$.

\subsection{Multiplication groups}

For an $r$-ary quasigroup $(G,q)$, the multiplication group
$\Mlt(G,q)\le\operatorname{Sym}(G)$ is generated by the permutations obtained
by fixing $r-1$ inputs of $q$ and varying the remaining one.  Write
$L_a(x)=ax$, $R_b(x)=xb$, and
$\cD(G)=L(G)R(G)=\langle L(G),R(G)\rangle$.

\begin{theorem}[Multiplication group]\label{thm:multiplication-group}
Let $f(x)=a\phi(x)$ with $a\in G$ and $\phi\in\Aut(G)$, and define
$\omega_f(x_1,\ldots,x_r)=f(x_1\cdots x_r)$.  Then
\begin{equation}\label{eq:mlt-structure}
  \Mlt(G,\omega_f)=\cD(G)\langle\phi\rangle.
\end{equation}
The subgroup $\cD(G)$ is normal, and
\begin{equation}\label{eq:mlt-out}
  \Mlt(G,\omega_f)/\cD(G)
  \cong\langle[\phi]\rangle\le\Out(G).
\end{equation}
The stabilizer of the identity element is
\begin{equation}\label{eq:mlt-stabilizer}
  \Mlt(G,\omega_f)_e=\langle\Inn(G),\phi\rangle\le\Aut(G).
\end{equation}
\end{theorem}

\begin{proof}
A one-coordinate translation has the form
$x\mapsto f(AxB)=fL_AR_B(x)$ for suitable $A,B\in G$.  Taking all fixed
inputs equal to $e$ shows that $f$ itself is one of these permutations.
Varying the first or last fixed block and composing with $f^{-1}$ then gives
all left and right regular translations.  Hence
$\Mlt(G,\omega_f)=\langle f,L(G),R(G)\rangle$.  Since
$f=L_a\phi$ and $L_a\in\cD(G)$, this is
$\cD(G)\langle\phi\rangle$.  Automorphisms normalize $L(G)$ and $R(G)$, so
$\cD(G)$ is normal.

If $L_aR_b$ is an automorphism, it fixes $e$, so $ab=e$ and
$L_aR_b=\Ad_a$.  Therefore
$\cD(G)\cap\Aut(G)=\Inn(G)$.  The quotient and stabilizer formulas follow.
\end{proof}

For a translation, $\phi=\id$, so the multiplication group is
$\cD(G)$.  For $\omega(x_1,\ldots,x_r)=d\psi(x_1\cdots x_r)$,
\eqref{eq:mlt-out} becomes
$\Mlt(G,\omega)/\cD(G)\cong\langle[\psi]\rangle\le\Out(G)$.
The classifications of individual trees and of tree-pair identities depend
on the order of $\psi$ in $\Aut(G)$, whereas this multiplication-group
quotient depends on the order of the class $[\psi]$ in $\Out(G)$.  Thus the two constructions record
different orders when a nontrivial power of $\psi$ is inner.  For example,
let $G=S_3$ and $\psi=\Ad_{(12)}$.  Then $\psi$ has order $2$ in
$\Aut(S_3)$ but $[\psi]=1$ in $\Out(S_3)$.  For the binary operation
$\omega(x,y)=\psi(xy)$, the reassociation theorem therefore gives a copy of
$F_4$, whereas $\Mlt(G,\omega)/\cD(G)$ is trivial.

\subsection{Recovering the group product on a perfect tree}

Let $P_{r,h}$ be the perfect full $r$-ary tree of height $h$.  We now ask for
the stronger equality
\[
  \omega_{P_{r,h}}(x_1,\ldots,x_{r^h})=x_1\cdots x_{r^h},
\]
with no final bijection.

\begin{corollary}[Operations that recover the product on $P_{r,h}$]
\label{cor:perfect-roots}
Let $r,h\ge2$.  The displayed equality holds for all inputs if and only if
\[
  \omega(x_1,\ldots,x_r)=d\psi(x_1\cdots x_r)
\]
for unique $d\in Z(G)$ and $\psi\in\Aut(G)$ satisfying
\begin{equation}\label{eq:perfect-root-conditions}
  \psi^h=\id_G,
  \qquad
  \prod_{j=0}^{h-1}\psi^j(d)^{r^j}=e.
\end{equation}
\end{corollary}

\begin{proof}
The perfect tree contains a height-two fringe in which all $r$ children are
nontrivial, so \cref{prop:height-two} forces the third form in
\eqref{eq:intro-three-families}.  Every leaf has depth $h$, so equality of
the variable coefficients gives $\psi^h=\id$.  There are $r^j$ internal
vertices at depth $j$, and \eqref{eq:central-factor} therefore gives the
second condition.  The converse follows from \cref{prop:affine-formula}.
\end{proof}

Write the abelian group $Z(G)$ additively.  For a fixed $\psi$ with
$\psi^h=\id$, set $\Theta=r\psi\in\End(Z(G))$ and
$S_h(\Theta)=1+\Theta+\cdots+\Theta^{h-1}$.

\begin{theorem}[The central constant]\label{thm:central-norm}
The constants $d\in Z(G)$ satisfying the second condition in
\eqref{eq:perfect-root-conditions} are exactly $\ker S_h(\Theta)$.  Moreover,
\begin{equation}\label{eq:geometric-identity}
  (\Theta-1)S_h(\Theta)=(r^h-1)1_{Z(G)},
\end{equation}
so
\begin{equation}\label{eq:norm-torsion}
  \ker S_h(\Theta)\subseteq Z(G)[r^h-1],
\end{equation}
where $Z(G)[m]=\{z\in Z(G):mz=0\}$.  In particular, $d=e$ is the only
possible constant whenever multiplication by $r^h-1$ is injective on
$Z(G)$; this includes torsion-free centers and centers of finite exponent
coprime to $r^h-1$.
\end{theorem}

\begin{proof}
The second condition in \eqref{eq:perfect-root-conditions} becomes, in
additive notation,
\[
  \sum_{j=0}^{h-1}r^j\psi^j(d)=S_h(\Theta)d=0.
\]
This gives the kernel statement.  The polynomial identity
$(X-1)(1+X+\cdots+X^{h-1})=X^h-1$ and the equality $\psi^h=1$ give
\eqref{eq:geometric-identity}.  If $S_h(\Theta)d=0$, applying $\Theta-1$
gives $(r^h-1)d=0$, which proves \eqref{eq:norm-torsion}.
\end{proof}

\begin{example}[A family of exact roots]\label{ex:c7-roots}
Let $G=C_7$, let $r=2$, and let $\psi(x)=x^2$.  The automorphism $\psi$ has
order $3$.  For every $d\in C_7$, define
$\omega_d(x,y)=d(xy)^2$.  On the perfect binary tree of height $3$, the
variable coefficient is $\psi^3=\id$, while the central factor is
\[
  d\,\psi(d)^2\,\psi^2(d)^4=d^{1+4+16}=d^{21}=e.
\]
Hence every one of the seven operations $\omega_d$ has the ordinary eight-fold
product as its iterate on $P_{2,3}$.  Here $r^h-1=7$, so the torsion allowed by
\cref{thm:central-norm} is attained.
\end{example}

\begin{remark}[Cyclic cohomology]\label{rem:H1}
If the exponent of $Z(G)$ divides $r^h-1$, then $\Theta^h=1$ on $Z(G)$, so
$\Theta$ defines an action of the cyclic group $C_h$ of order $h$.  In this case
$\ker S_h(\Theta)=Z^1(C_h,Z(G))$.  Conjugating the $r$-ary operation by the
central translation $x\mapsto zx$ changes $d$ by the coboundary
$d\mapsto d-z+\Theta z$ in additive notation.  The resulting conjugacy
classes are therefore parametrized by $H^1(C_h,Z(G))$; see Brown
\cite{BrownCohomology1982}.  This remark is only an interpretation of the
kernel above.
\end{remark}

\section{Discussion}\label{sec:discussion}

The paper begins with a fixed tree $T$ and an arbitrary operation $\omega$.  Requiring one iterate $\omega_T$ to depend only on the ordered group product is already rigid: $\omega$ must be a left translation of the product, a right translation, or an operation $d\psi(x_1\cdots x_r)$ with $d\in Z(G)$ and $\psi\in\Aut(G)$.  The tree determines which parameters in these families are allowed.

We then reverse the problem.  Once one of these operations is fixed, the explicit tree formulas classify every $T$ for which there is a bijection $F_T$ satisfying $\omega_T(x_1,\ldots,x_n)=F_T(x_1\cdots x_n)$.  For translations the answer depends on the order of $aZ(G)$ or $bZ(G)$ in $G/Z(G)$; for the automorphism form it depends on the order of $\psi$.  Finite order permits periodic local insertion of fixed vine or perfect-tree blocks, whereas infinite order forces a single global shape.

We also studied a separate, related question: for each of these three solution forms, how much associativity remains?  We compared arbitrary same-size trees $S$ and $T$, without requiring either tree to satisfy the product equation, and classified all identities $\omega_S=\omega_T$.  These parenthesization changes form a subgroup of the Higman--Thompson group $F_r$.  The same group-theoretic periods that govern the individual-tree problem determine these reassociation groups: finite quotient order $k$ gives a vine copy of $F_{1+k(r-1)}$, finite automorphism order $h$ gives a perfect-tree copy of $F_{r^h}$, and infinite order gives the trivial subgroup.

The crossed-module interpretation explains why the three algebraic forms arise, while the final section records consequences for multiplication groups and for exact recovery of the group product on perfect trees.  We do not address identities of arbitrary $r$-ary operations unrelated to the initial group-product equation, nor arbitrary embeddings into $F_r$.

\section*{AI Disclosure Statement}
OpenAI ChatGPT (GPT-5.6 Sol, August 2026) was used for mathematical
exploration, proof checking, literature organization, and language editing.
The author verified the resulting statements, proofs, calculations, and
references and takes full responsibility for the manuscript.

\end{document}